\documentclass[oneside, 12pt]{amsart}

\usepackage{comment}

\usepackage{mathtools, nccmath}

\usepackage{enumerate}
\usepackage{array}
\usepackage[utf8]{inputenc}
\usepackage[margin=1in]{geometry}
\usepackage{amsthm}
\usepackage{enumitem}  
\usepackage[OT2,T1]{fontenc}
\usepackage{amscd,amssymb, enumitem, amsmath,mathrsfs, amsfonts, amsaddr}
\usepackage{url}
\usepackage{tikz}
\usepackage{fullpage}
\usepackage{microtype}
\usetikzlibrary{decorations.pathreplacing}
\usepackage[backref=page]{hyperref}
\usepackage{hyperref}
\usepackage[utf8]{inputenc}
\usepackage{enumitem}
\usepackage[new]{old-arrows}
\usepackage{mathtools}

\DeclareSymbolFont{cyrletters}{OT2}{wncyr}{m}{n}
\DeclareMathSymbol{\Sha}{\mathalpha}{cyrletters}{"58}

\makeatletter
\newcommand{\pushright}[1]{\ifmeasuring@#1\else\omit\hfill$\displaystyle#1$\fi\ignorespaces}
\newcommand{\pushleft}[1]{\ifmeasuring@#1\else\omit$\displaystyle#1$\hfill\fi\ignorespaces}
\makeatother

\usepackage[utf8]{inputenc}

\newtheorem{theorem}{Theorem}[section]

\newtheorem{proposition}[theorem]{Proposition}
\newtheorem*{proposition*}{Proposition}

\newtheorem*{theorem*}{Theorem}
\newtheorem*{corollary*}{Corollary}

\theoremstyle{definition}

\newtheorem{example}[theorem]{Example}

\newtheorem{claim}[theorem]{Claim}
\newtheorem{remark}[theorem]{Remark}
\newtheorem{emp}[theorem]{\bf }{\kern -4pt}
\newtheorem*{acknowledgement}{Acknowledgements}

\theoremstyle{theorem}
\newtheorem{conjecture}[theorem]{Conjecture}
\theoremstyle{plain}

\newtheoremstyle{TheoremNum}
{\topsep}{\topsep}              
{\itshape}                      
{}                              
{\bfseries}                     
{.}                             
{ }                             
{\thmname{#1}\thmnote{ \bfseries #3}}
\theoremstyle{TheoremNum}

\newtheorem{prorep}{Proposition}

\theoremstyle{remark}

\newcommand{\legendre}[2]{\ensuremath{\left( \frac{#1}{#2} \right) }}

\newcommand{\bd}{400}
\newcommand{\repisogeny}{Section 3}

\title{Congruence properties of the coefficients of the classical modular polynomials}
\author{Haiyang Wang}
\address{}
\date{\today}

\usepackage{draftwatermark}
\SetWatermarkText{}
\SetWatermarkScale{3}

\begin{document}
	
	\maketitle
	

	\begin{abstract}
		The classical modular polynomials $\Phi_\ell(X,Y)$ give plane curve models for the modular curves $X_0(\ell)/\mathbb{Q}$ and have been extensively studied. In this article, we provide closed formulas for $\ell$ nontrivial coefficients of the classical modular polynomials $\Phi_\ell(X,Y)$ in terms of the Fourier coefficients of the modular invariant function $j(z)$ for a prime $\ell$. Our interest in the formulas were motivated by our conjectures on congruences modulo powers of the primes $2,3$ and $5$ satisfied by the coefficients of these polynomials. We deduce congruences from these formulas supporting the conjectures.
	\end{abstract}



\section{Introduction}

Let $\ell$ be a prime. Let $\Phi_\ell(X,Y)\in\mathbb{Z}[X,Y]$ be the $\ell$-th classical modular polynomial. Recall that
\begin{equation*}
	\Phi_\ell(X,Y)=X^{\ell+1}+Y^{\ell+1}+\sum\limits_{0\le m,n\le \ell}a_{m,n}X^mY^n,
\end{equation*}
for some $a_{m,n}\in\mathbb{Z}$ with $a_{m,n}=a_{n,m}$. See \cite{Lang}, Ch. 5, Theorem 3 and \cite{SilvermanAd} Theorem II.6.3.

Let $j(z)$ be the modular invariant function (see \cite{SilvermanAd} p. 34). Let $q:=e^{2\pi iz}$ and let
\begin{equation}\label{eq.qexpansion}
	j(z)=\sum\limits_{i=-1}^\infty c_i q^i=\frac{1}{q}+744+196884q+\cdots
\end{equation}
be the Fourier expansion of $j(z)$.	
In this paper, we provide the following closed formulas for $\ell$ nontrivial coefficients of $\Phi_\ell(X,Y)$ in terms of the Fourier coefficients of $j(z)$. Recall that it is well known that $a_{\ell,\ell}=-1$ (see \cite{Lang} Theorem 3, iii) and p. 57).

\medskip

\begin{theorem}\label{thm.form}
	
	Let $\ell\ge 3$ be a prime and $0< m\le\ell$. Let $a_{\ell,\ell-m}$ be the coefficient of $\Phi_\ell(X,Y)$.

	\begin{enumerate}[label=(\roman*)]
		\item\label{it.el_m} 	If $0<m< \ell$, then
		\begin{equation}\label{eq.a_l_l_m}
			a_{\ell,\ell-m}=\sum\limits_{\substack{ t_1 r_1+\cdots +t_\lambda r_\lambda=m\\
					0<r_1<\cdots <r_\lambda,\; 1\le t_i\\
					u:=t_1+\cdots +t_\lambda-1}} (-1)^u\frac{u!}{t_1!\cdots t_\lambda!}\ell\binom{\ell-m+u}{u}c_{r_1-1}^{t_1}\cdots c_{r_\lambda-1}^{t_\lambda}.
		\end{equation}

		\item\label{it.el_0} If $m=\ell$, then 
		\begin{equation}\label{eq.a_l_0}
			a_{\ell,0}=-(\ell+1)c_0+\sum\limits_{\substack{ t_1 r_1+\cdots +t_\lambda r_\lambda=\ell\\
					0<r_1<\cdots <r_\lambda,\; 1\le t_i\\
					u:=t_1+\cdots +t_\lambda-1}} (-1)^u\frac{u!}{t_1!\cdots t_\lambda!}\ell c_{r_1-1}^{t_1}\cdots c_{r_\lambda-1}^{t_\lambda}.
		\end{equation}
		\item\label{it.int} Let $0<m\le \ell$. Assume that $0<r_1<\cdots <r_\lambda$ and $t_i\ge 1$ for $1\le i\le\lambda$ satisfying $ t_1 r_1+\cdots +t_\lambda r_\lambda=m$. Let $ 
		u:=t_1+\cdots +t_\lambda-1$. Then $\frac{u!}{t_1!\cdots t_\lambda!}\ell\binom{\ell-m+u}{u}$ is an integer.
	\end{enumerate}

\end{theorem}

\medskip

The proof of Theorem \ref{thm.form} is found in Section \ref{em.proof_mod}. Yui introduced an algorithm for computing $\Phi_\ell(X,Y)$ using the coefficients of the $j$-function in \cite{Yui}. However, the paper does not deduce closed formulas for the coefficients of $\Phi_\ell(X,Y)$. Different algorithms for computing $\Phi_\ell(X,Y)$ are found in \cite{BLS}, \cite{Charles-Lauter}, and \cite{Ito_Alg}.

Our interest in the formulas was motivated by our conjecture below on congruences modulo powers of the primes $2,3$ and $5$ satisfied by the coefficients of $\Phi_\ell(X,Y)$.

\begin{conjecture}\label{conj.div}
	Let $\ell$ be a prime and let $a_{m,n}$ be the coefficient of $\Phi_\ell(X,Y)$. Assume that $\ell+1>m+n$ and let $c:=\ell+1-m-n$. Then the following is true.
	\begin{enumerate}[label=(\alph*)]
		\item If $\ell\ne 2$, then $a_{m,n}\equiv 0 \operatorname{mod} 2^{15c}$.
		\item If $\ell\ne 3$, then $a_{m,n}\equiv 0\operatorname{mod} 3^{3c}$. Moreover, if $\ell\equiv 1\operatorname{mod} 3$, then
		$a_{m,n}\equiv 0\operatorname{mod} 3^{\lceil\frac{9}{2}c\rceil}$.
		\item If $\ell\ne 5$, then $a_{m,n}\equiv 0\operatorname{mod} 5^{3c}$.
	\end{enumerate}
\end{conjecture}

\medskip

Related congruences for the coefficients of $\Phi_\ell(X,Y)$ of the form $a_{\ell,\ell-m}$ are found in Proposition \ref{prop.2Div}, Proposition \ref{prop.3Div}, and Conjecture \ref{conj.a_l_m_5Div}.

We have verified that Conjecture \ref{conj.div} is true for all primes $\ell< 400$ (see Proposition \ref{prop.verif}), through examining the explicit computations of the modular polynomials obtained by Sutherland \cite{Sutherland}.

\begin{remark}
	Computations suggest that Conjecture \ref{conj.div} is sharp. For instance, computations suggest that for any given $\ell$, there exists $m,n\ge 0$ with $c:=\ell+1-m-n>0$ such that $\operatorname{ord}_2(a_{m,n})=15c$.
\end{remark}

Theorem \ref{thm.form} can be applied to yield the following congruences, which provide evidence supporting Conjecture \ref{conj.div} when taking $m=\ell$.
\medskip

\begin{prorep}[\ref{prop.2Div}]
	Let $\ell$ be an odd prime and let $0<m\le \ell$ be an integer. Let $a_{\ell,\ell-m}$ be the coefficient of the $\ell$-th classical modular polynomial $\Phi_\ell(X,Y)$. Then

	\begin{enumerate}[label=(\roman*)]
		\item If $m\equiv 4\operatorname{mod} 8$, then $a_{\ell,\ell-m}\equiv 0\operatorname{mod} 2$.
		\item If $m\equiv 2\operatorname{mod} 4$, then $a_{\ell,\ell-m}\equiv 0\operatorname{mod} 2^2$.
		\item If $m\equiv 1\operatorname{mod} 8$, then $a_{\ell,\ell-m}\equiv 0\operatorname{mod} 2^3$.
		\item If $m\equiv 5\operatorname{mod} 8$, then 
		$a_{\ell,\ell-m}\equiv 0\operatorname{mod} 2^4$.
		\item If $m\equiv 3\operatorname{mod} 4$, then $a_{\ell,\ell-m}\equiv 0\operatorname{mod} 2^5$. 
	\end{enumerate}

\end{prorep}

\medskip

\begin{prorep}[\ref{prop.3Div}]
	Let $\ell$ be a prime and let $0<m\le \ell$ be an integer. Let $a_{\ell,\ell-m}$ be the coefficient of the $\ell$-th classical modular polynomial $\Phi_\ell(X,Y)$. Then
	\begin{enumerate}[label=(\roman*)]
		\item If $m\equiv 1\operatorname{mod} 3$, then $a_{\ell,\ell-m}\equiv 0\operatorname{mod} 3$.
		\item If $m\equiv 2\operatorname{mod} 3$, then 
		$a_{\ell,\ell-m}\equiv 0\operatorname{mod} 3^2$.
	\end{enumerate}
\end{prorep}

\begin{remark}
	Computations seem to suggest that in general $a_{\ell,\ell-m}$ is not divisible by $2$ when $m\equiv 0\operatorname{mod} 8$ and not divisible by $3$ when $m\equiv 0\operatorname{mod} 3$.
\end{remark}

Further evidence for Conjecture \ref{conj.div} is provided by Proposition \ref{prop.a_0_0_2}, which proves in particular that the coefficient $a_{0,0}$ of $\Phi_\ell(X,Y)$ is always divisible by $2^{15}3^3 5^3$ for any prime $\ell>5$.

Our initial interest in formulating Conjecture \ref{conj.div} came from considerations about elliptic curves and is discussed briefly at the end of Section \ref{sec.modularpoly}.

\begin{acknowledgement}
	Much of this work is contained in the author’s doctoral thesis at the University of Georgia. The author would like to thank his advisor, Dino Lorenzini, for his very helpful guidance and feedback. The author also thanks Pete Clark for his suggestions on Propositions \ref{prop.a_0_0} and \ref{prop.a_0_0_2}.
\end{acknowledgement}

\section{Coefficients of the modular polynomials}\label{sec.modularpoly}

In this section, we prove Theorem \ref{thm.form} and deduce some congruences satisfied by the coefficients of the modular polynomials from the theorem.

\begin{proposition}\label{prop.verif}
	Conjecture \ref{conj.div} is true for $\ell\le \bd$.
\end{proposition}

\begin{proof}
	
	By applying the algorithms in \cite{BLS} and \cite{Bruinier-Ono-Sutherland}, Sutherland computed the modular polynomials $\Phi_N(X,Y)$ for small $N$ \cite{Sutherland}. In particular, he computed $\Phi_N(X,Y)$ for $N\le \bd$. Note that the coefficients of $a_{m,n}$ are often of massive size. For instance, the constant coefficient of $\Phi_5(X,Y)$ is divisible already by $2^{90}$. By checking the coefficients of modular polynomials from \cite{Sutherland}, we verified that Conjecture \ref{conj.div} is true for all primes $\ell\le 400$. Our computations were done using Magma \cite{Magma}.
\end{proof}

\begin{remark}
	We briefly discuss here the history of computing $\Phi_\ell(X,Y)$ for prime $\ell$. 
	The polynomial $\Phi_2(X,Y)$ was given explicitly in 1922 by Fricke in his book \cite{Fricke} on page $372$.
	We could not find the first reference for $\Phi_2(X,Y)$. The case $\ell=3$ was done by Smith in 1878 \cite{Smith}. It seems fair to say that the case $\ell=2$ must have been known to Smith at the time. The case $\ell=5$ was due to Berwick (\cite{Berwick}, 1916). It is interesting to note that the case $\ell=7$ seems to have been computed only in 1975 by Herrmann \cite{Herrmann}. Note that Herrmann seemed to have been unaware of Berwick's computation of the case $\ell=5$ in 1916. The case $\ell=11$ was obtained by Kaltofen and Yui (\cite{Kaltofen-Yui}, 1984). More recently, Sutherland has computed $\Phi_\ell(X,Y)$ for all primes $\ell<1000$ (see \cite{Sutherland}) using the algorithm in \cite{BLS}. Due to the generally massive size of the coefficients of $\Phi_\ell(X,Y)$, we only verified Conjecture \ref{conj.div} up to $400$ using Sutherland's computations. 
\end{remark}

To illustrate Conjecture \ref{conj.div}, we list the factored coefficients of $\Phi_\ell(X,Y)$ for $\ell=5$ (see \cite{Herrmann} page 193, or \cite{Sutherland}).
{\allowdisplaybreaks
	\begin{align*}
		a_{0,0}&=2^{\bf 90}\cdot 3^{\bf 18}\cdot 5^3\cdot 11^9\\
		a_{0,1}&=2^{\bf 77}\cdot 3^{\bf 16}\cdot 5^3\cdot 11^6\cdot 31\cdot 1193\\
		a_{1,1}&=-2^{\bf 62}\cdot 3^{\bf 13}\cdot 11^3\cdot 26984268714163\\
		a_{0,2}&=2^{\bf 60}\cdot 3^{\bf 13}\cdot 5^{2}\cdot 11^{3}\cdot 13^{2}\cdot 3167\cdot 204437\\
		a_{1,2}&=2^{\bf 47}\cdot 3^{\bf 10}\cdot 5^4\cdot 53359\cdot 131896604713\\
		a_{2,2}&=2^{\bf 30}\cdot 3^{\bf 8}\cdot 5^4\cdot 7\cdot 13\cdot 1861\cdot 6854302120759\\
		a_{0,3}&=2^{\bf 48}\cdot 3^{\bf 9}\cdot 5^2\cdot 31\cdot 1193\cdot 24203\cdot 2260451\\
		a_{1,3}&=-2^{\bf 31}\cdot 3^{\bf 7}\cdot 5^3\cdot 327828841654280269\\
		a_{2,3}&=2^{\bf 17}\cdot 3^{\bf 4}\cdot 5^3\cdot 2311\cdot 2579\cdot 3400725958453\\
		a_{3,3}&=-2^2\cdot 5^2\cdot 11\cdot 17\cdot 131\cdot 1061\cdot 169751677267033\\
		a_{0,4}&=2^{\bf 30}\cdot 3^{\bf 7}\cdot 5\cdot 13^2\cdot 3167\cdot 204437\\
		a_{1,4}&=2^{\bf 20}\cdot 3^{\bf 4}\cdot 5^3\cdot 12107359229837\\
		a_{2,4}&=3\cdot 5^3\cdot 167\cdot 6117103549378223\\
		a_{3,4}&=2^5\cdot 3\cdot 5^2\cdot 197\cdot 227\cdot 421\cdot 2387543\\
		a_{4,4}&=2^3\cdot 5^2\cdot 257\cdot 32412439\\
		a_{0,5}&=2^{\bf 17}\cdot 3^{\bf 4}\cdot 5\cdot 31\cdot 1193 \\
		a_{1,5}&=-2\cdot 3\cdot 5^2\cdot 1644556073\\
		a_{2,5}&=2^5\cdot 5^2\cdot 13\cdot 195053\\
		a_{3,5}&=-2^2\cdot 3^2\cdot 5\cdot 131\cdot 193\\
		a_{4,5}&=2^3\cdot 3\cdot 5\cdot 31\\
\end{align*}}%
\begin{emp}\label{em.proof_mod}
	
	Now we will give the proof of Theorem \ref{thm.form}.

	\begin{proof}[Proof of Theorem \ref{thm.form}]
		Fix a prime $\ell\ge 3$. From \eqref{eq.qexpansion}, we know that 
		\[j(\ell z)=\sum\limits_{i=-1}^\infty c_i q^{\ell i}.\]
		
		By construction, $\Phi_\ell(X,Y)=X^{\ell+1}+Y^{\ell+1}+\sum\limits_{0\le r,s\le \ell}a_{r,s}X^rY^s,$ and
		
		\begin{equation}\label{eq.construction}
			\Phi_\ell(j(\ell z), j(z))=0.
		\end{equation}

		We now prove \ref{thm.form} \ref{it.el_m}. Observe that \eqref{eq.a_l_l_m} is equivalent to the following:
		\begin{equation}\label{eq.a_l_l_m_alt}
			a_{\ell,\ell-m}=\sum\limits_{\substack{ t_1+t_2 2+\cdots +t_{m} m=m\\
					0\le t_i,\;
					u:=t_1+\cdots +t_m-1}} (-1)^{u}\frac{u!}{\prod\limits_{1\le i\le m} t_i ! }\ell\binom{\ell-m+u}{u}c_{0}^{t_1}\cdots c_{m-1}^{t_m}.
		\end{equation}

		We prove \eqref{eq.a_l_l_m_alt} by induction on $m$. Since $a_{\ell,\ell}=-1$, it is immediate that \eqref{eq.a_l_l_m_alt} holds for $m=0$. Assume that \eqref{eq.a_l_l_m_alt} is true for $a_{\ell,\ell-1},\dots,a_{\ell,\ell-m+1}$ for some $m-1$ satisfying $1\le m-1\le \ell-2$. In the following, we show that it is true for $a_{\ell,\ell-m}$. Recall the notion of multinomial coefficients. Let $f\in\mathbb{Z}_{\ge 1}$ and $g_1, \dots, g_\eta\in\mathbb{Z}_{\ge 0}$ satisfying $\sum_{1\le i\le \eta} g_i=f$. Define $\binom{f}{g_1,\cdots,g_\eta}:=\frac{f!}{g_1!\cdots g_\eta !}$.

		\begin{claim}\label{cl.coeff3}
			The coefficient of the term $q^{-\ell^2-\ell+m}$ in the $q$-expansion of the left-hand side of \eqref{eq.construction}, i.e. $j(\ell z)^{\ell+1}+j(z)^{\ell+1}+\sum\limits_{0\le r,s\le \ell}a_{r,s}j(\ell z)^r j(z)^s$, is

			\begin{equation}\label{eq.coeff2}
				a_{\ell,\ell-m}+\sum_{0\le n\le m-1} a_{\ell,\ell-n}\left(\sum\limits_{\substack{t_1+t_2 2+\cdots+t_{m-n}(m-n)=m-n\\
						t_i\ge 0}}\binom{\ell-n}{t_1,\dots,t_{m-n},\ell-n-\sum\limits_{1\le i\le m-n} t_i}c_0^{t_1}\cdots c_{m-n-1}^{t_{m-n}}\right).
			\end{equation}
		\end{claim}

		Indeed, the coefficient of the term $q^{-\ell^2-\ell+m}$ in the $q$-expansion of $j(\ell z)^{\ell+1}$ is 0 because all the powers in the $q$-expansion of $j(\ell z)$ are divisible by $\ell$ and $1\le m\le\ell-1$. 
		
		Assume that $0\le r\le \ell$ and $0\le s\le\ell+1$ satisfy one of the following three conditions: (a) $r=0$ and $s=\ell+1$; (b) $0\le r<\ell$ and $0\le s\le \ell$; (c) $r=\ell$ and $0\le s<\ell-m$. Then the lowest degree in the $q$-expansion of $a_{r,s}j(\ell z)^rj(z)^s$ is greater than $-\ell^2-\ell+m$. 
		
		The coefficient of the term $q^{-\ell^2-\ell+m}$ in the $q$-expansion of $a_{\ell,\ell-m}j(\ell z)^\ell j(z)^{\ell-m}$ is $a_{\ell,\ell-m}$.

		\begin{claim}\label{cl.A_n}
			Let $0\le n< m$ and let $A_n$ denote the coefficient of the term $q^{-\ell^2-\ell+m}$ in the $q$-expansion of $j(\ell z)^\ell j(z)^{\ell-n}$. Then
			$$A_n=\sum\limits_{\substack{t_1+t_2 2+\cdots+t_{m-n}(m-n)=m-n\\
					t_i\ge 0}}\binom{\ell-n}{t_1,\dots,t_{m-n},\ell-n-\sum\limits_{1\le i\le m-n} t_i}c_0^{t_1}\cdots c_{m-n-1}^{t_{m-n}}.$$
		\end{claim}

		Because $j(z)=\sum\limits_{i=-1}^\infty c_i q^i=q^{-1}\sum\limits_{i=0}^\infty c_{i-1} q^i$ and $j(\ell z)=\sum\limits_{i=-1}^\infty c_i q^{\ell i}=q^{-\ell}\sum\limits_{i=0}^\infty c_{i-1} q^{\ell i}$, we have that

		$$j(\ell z)^\ell j(z)^{\ell-n}=q^{-\ell^2-\ell+n}(\sum\limits_{i=0}^\infty c_{i-1} q^{\ell i})^\ell(\sum\limits_{i=0}^\infty c_{i-1} q^i)^{\ell-n}.$$

		In the $q$-expansion of $(\sum\limits_{i=0}^\infty c_{i-1} q^{\ell i})^\ell$, the lowest degree term is $1$, and the second lowest degree is $\ell$. So $A_n$ equals the coefficient of $q^{m-n}$ in the $q$-expansion of $(\sum\limits_{i=0}^\infty c_{i-1} q^i)^{\ell-n}$.

		Let $j_1(z)=\sum\limits_{i=0}^{m-n} c_{i-1} q^i$ and $j_2(z)=\sum\limits_{i=m-n+1}^\infty c_{i-1} q^i$. Then $(\sum\limits_{i=0}^\infty c_{i-1} q^i)^{\ell-n}=(j_1(z)+j_2(z))^{\ell-n}=\sum\limits_{k=0}^{\ell-n}\binom{\ell-n}{k} (j_1(z))^{\ell-n-k}(j_2(z))^{k}$.

		If $k>0$, then the lowest degree in the $q$-expansion of $j_2(z)^k$ is $k(m-n+1)>m-n$. So $A_n$ equals the coefficient of the term $q^{m-n}$ in the $q$-expansion of $(j_1(z))^{\ell-n}$. 
		
		{\allowdisplaybreaks
			\begin{align*}
				(j_1(z))^{\ell-n}&=(\sum\limits_{i=0}^{m-n} c_{i-1} q^i)^{\ell-n}\\
				&\stackrel{(*)}{=}\sum\limits_{\substack{t_1+\cdots+t_{m-n+1}=\ell-n\\
						t_i\ge 0}} \binom{\ell-n}{t_1,\dots,t_{m-n},t_{m-n+1}}(c_0 q)^{t_1}(c_1 q^2)^{t_2}\cdots(c_{m-n-1}q^{m-n})^{t_{m-n}}1^{t_{m-n+1}}\\
				&=\sum\limits_{\substack{t_1+\cdots+t_{m-n}\le\ell-n\\
						t_i\ge 0}}\binom{\ell-n}{t_1,\dots,t_{m-n},\ell-n-\sum\limits_{1\le i\le m-n} t_i}c_0^{t_1}\cdots c_{m-n-1}^{t_{m-n}}q^{t_1+t_2 2+\cdots+t_{m-n}(m-n)}.
		\end{align*}}%
		Note that $(*)$ is by the Multinomial Theorem. 
		It follows that Claim \ref{cl.A_n} is true. So the coefficient of the term $q^{-\ell^2-\ell+m}$ in the $q$-expansion of $j(\ell z)^{\ell+1}+j(z)^{\ell+1}+\sum\limits_{0\le r,s\le \ell}a_{r,s}j(\ell z)^r j(z)^s$ is \eqref{eq.coeff2}. Therefore Claim \ref{cl.coeff3} is true.

		Since $\Phi_\ell(j(\ell z), j(z))=0$ in \eqref{eq.construction}, we find that the expression \eqref{eq.coeff2} equals $0$. Therefore,
		{\allowdisplaybreaks
			\begin{equation}\label{eq.ind}
				a_{\ell,\ell-m}=\enspace-\sum_{0\le n\le m-1} a_{\ell,\ell-n}\left(\sum\limits_{\substack{t_1+t_2 2+\cdots+t_{m-n}(m-n)=m-n\\
						t_i\ge 0}}\binom{\ell-n}{t_1,\dots,t_{m-n},\ell-n-\sum\limits_{1\le i\le m-n} t_i}c_0^{t_1}\cdots c_{m-n-1}^{t_{m-n}}\right).
		\end{equation}}%
		
		Our next step is to use the induction hypothesis to express $a_{\ell,\ell-n}$ when $n\le m-1$. For convenience, let us introduce the following notation. 
		
		For any $\lambda>0$ and $0\le r_1<r_2<\cdots<r_\lambda\le \ell$ and $0\le t_1,\dots,t_\lambda\le \ell$ such that $0\le \sum\limits_{1\le i\le\lambda} t_i r_i\le \ell$, we define
		\begin{equation*}
			d_{r_1,\dots,r_\lambda}^{t_1,\dots,t_\lambda}:=\left((-1)^{-1+\sum\limits_{1\le i\le\lambda} t_i}\right)\left(\frac{1}{\prod\limits_{1\le i\le\lambda} t_i ! }\right)\ell\left(\frac{(\ell-1-\sum\limits_{1\le i\le\lambda} t_i r_i+\sum\limits_{1\le i\le\lambda} t_i)!}{(\ell-\sum\limits_{1\le i\le\lambda} t_i r_i)!}\right),
		\end{equation*}
		so that \eqref{eq.a_l_l_m_alt} becomes
		\begin{equation}\label{eq.formsimp}
			a_{\ell,\ell-m}=\sum\limits_{\substack{ t_1+t_2 2+\cdots +t_{m} m=m\\
					0\le t_i,\;
					u:=t_1+\cdots +t_m-1}} d_{1,\dots,m}^{t_1,\dots,t_m}c_{0}^{t_1}\cdots c_{m-1}^{t_m}.
		\end{equation}

		Given $t_{i,1}$ with $0\le t_{i,1}\le t_i$, let $t_{i,2}:=t_i-t_{i,1}$. Then 
		{\allowdisplaybreaks
			\begin{align*}
				&a_{\ell,\ell-m}\\
				&\qquad=\enspace-\sum_{0\le n\le m-1}\left(\sum\limits_{\substack{ \overline{t}_1+\overline{t}_2 2+\cdots +\overline{t}_{m} m=n\\
						0\le \overline{t}_i,\;
						u:=\overline{t}_1+\cdots +\overline{t}_m-1}} (-1)^{u}\frac{u!}{\prod\limits_{1\le i\le m} \overline{t}_i ! }\ell\binom{\ell-n+u}{u}c_{0}^{\overline{t}_1}\cdots c_{m-1}^{\overline{t}_m}\right)\cdot\\
				&\pushright{\left(\sum_{\substack{ \hat{t}_1 +\hat{t}_2 2+\cdots +\hat{t}_m m=m-n\\
							0\le \hat{t}_i}} \binom{\ell-n}{\hat{t}_1,\dots,\hat{t}_m,\ell-n-\sum\limits_{1\le i\le m} \hat{t}_i}c_{0}^{\hat{t}_1}\cdots c_{m-1}^{\hat{t}_m}\right)\quad}\\
				&\qquad=\enspace-\sum_{0\le n\le m-1}\sum\limits_{\substack{ \overline{t}_1+\overline{t}_2 2+\cdots +\overline{t}_{m} m=n\\
						\hat{t}_1 +\hat{t}_2 2+\cdots +\hat{t}_m m=m-n\\
						\overline{t}_i\ge 0,\;\hat{t}_i\ge 0
				}} \left((-1)^{-1+\sum\limits_{1\le i\le m} \overline{t}_i}\right)\left(\frac{(-1+\sum\limits_{1\le i\le m} \overline{t}_i)!}{\prod\limits_{1\le i\le m} \overline{t}_i ! }\right)\ell\binom{\ell-1-n+\sum\limits_{1\le i\le m} \overline{t}_i}{-1+\sum\limits_{1\le i\le m} \overline{t}_i}\cdot\\
				&\pushright{ \binom{\ell-n}{\hat{t}_1,\dots,\hat{t}_m,\ell-n-\sum\limits_{1\le i\le m} \hat{t}_i}c_{0}^{\overline{t}_1+\hat{t}_1}\cdots c_{m-1}^{\overline{t}_m+\hat{t}_m}\quad}\\
				&\qquad=\enspace-\sum_{\substack{ t_1+t_2 2+\cdots +t_m m=m\\
						0\le t_i}}\sum_{\substack{ 0\le t_{\beta,1}\le t_\beta\\
						\text{for } 1\le \beta\le m\\
						(t_{1,1},\dots,t_{m,1})\ne \\
						(t_1,\dots,t_m)}} \left((-1)^{-1+\sum\limits_{1\le i\le m} t_{i,1}}\right)\left(\frac{(-1+\sum\limits_{1\le i\le m} t_{i,1})!}{\prod\limits_{1\le i\le m} t_{i,1} ! }\right)\ell\cdot\\
				&\pushright{\binom{\ell-1-\sum\limits_{1\le i\le m} t_{i,1} i+\sum\limits_{1\le i\le m} t_{i,1}}{-1+\sum\limits_{1\le i\le m}t_{i,1}}\binom{\ell-\sum\limits_{1\le i\le m} t_{i,1}i}{t_{1,2},\dots,t_{m,2},\ell-\sum\limits_{1\le i\le m} t_{i,1}i-\sum\limits_{1\le i\le m} t_{i,2}}c_{0}^{t_1}\cdots c_{m-1}^{t_m}\quad}\\
				&\qquad=\enspace-\sum_{\substack{ t_1+t_2 2+\cdots +t_m m=m\\
						0\le t_i}}\left(\sum_{\substack{ 0\le t_{\beta,1}\le t_\beta\\
						\text{for } 1\le \beta\le m\\
						(t_{1,1},\dots,t_{m,1})\ne \\
						(t_1,\dots,t_m)}} d_{1,\dots,m}^{t_{1,1},\dots,t_{m,1}}\binom{\ell-\sum\limits_{1\le i\le m} t_{i,1}r_i}{t_{1,2},\dots,t_{m,2},\ell-\sum\limits_{1\le i\le m} t_{i,1}r_i-\sum\limits_{1\le i\le m} t_{i,2}}\right)c_{0}^{t_1}\cdots c_{m-1}^{t_m}.
		\end{align*}}%

		To prove \eqref{eq.formsimp}, it is now enough to prove that
		{\allowdisplaybreaks
			\begin{equation}\label{eq.coeff}
				d_{r_1,\dots,r_\lambda}^{t_1,\dots,t_\lambda}=-\sum_{\substack{ 0\le t_{1,1}\le t_1\\
						\cdot\cdot\cdot\\
						0\le t_{\lambda,1}\le t_\lambda\\
						(t_{1,1},\dots,t_{\lambda,1})\ne \\
						(t_1,\dots,t_\lambda)}} d_{r_1,\dots,r_\lambda}^{t_{1,1},\dots,t_{\lambda,1}}\binom{\ell-\sum\limits_{1\le i\le \lambda} t_{i,1}r_i}{t_{1,2},\dots,t_{\lambda,2},\ell-\sum\limits_{1\le i\le \lambda} t_{i,1}r_i-\sum\limits_{1\le i\le \lambda} t_{i,2}}
		\end{equation}}%
		for any $\lambda>0$, $0<r_1<\cdots <r_\lambda$ and $t_1,\dots,t_\lambda\ge 1$ such that $\sum\limits_{1\le i\le\lambda} r_it_i=m$. The formula \eqref{eq.coeff} is proved below using Claim \ref{cl.recur}. First, we need the following.

		For $d,n\in\mathbb{N}$, let $\bold{s}(d,n)$ and $\bold{S}(d,n)$ denote the Stirling number of the first and of the second kind, respectively. Recall that $\bold{s}(d,n)$ is $(-1)^{d-n}$ multiplies the number of elements in the symmetric group $S_d$ with $n$ disjoints cycles and $\bold{S}(d,n)$ is the number of ways to partition a set of $d$ elements into $n$ non-empty subsets. If $0\le d<n$, then $\bold{S}(d,n)=0$ by the definition. The numbers $\bold{s}(d,n)$ and $\bold{S}(d,n)$ are characterized by Equations \eqref{eq.stirling1} and \eqref{eq.stirling2_0} below, respectively. 
		
		For $n \ge 1$, we have that 
		\begin{equation}\label{eq.stirling1}
			\prod\limits_{0\le i\le n-1} (x-i)=\sum\limits_{0\le i\le n} \bold{s}(n,i)x^i.
		\end{equation}
		See \cite{Riordan} p. 33 (34). For $d\ge 0$ and $n\ge 0$, we have that 
		\begin{equation}\label{eq.stirling2_0}
			\sum\limits_{0\le i\le n} (-1)^i\binom{n}{i} (n-i)^d= n!\bold{S}(d,n).
		\end{equation}
		See \cite{Riordan} p. 33 (38) and Example $12$ on p. $13$. By changing $i$ to $n-i$, we get that 
		\begin{equation}\label{eq.stirling2}
			\sum\limits_{0\le i\le n} (-1)^i\binom{n}{i} i^d=(-1)^n n!\bold{S}(d,n).
		\end{equation}

		For the rest of the proof of Theorem \ref{thm.form} \ref{it.el_m}, fix $\lambda>0$, $0<r_1<\cdots <r_\lambda$ and $t_1,\dots,t_\lambda\ge 1$ such that $\sum\limits_{1\le i\le\lambda} r_it_i=m$.

		\begin{claim}\label{cl.recur} Let $\lambda>0$ and let $\alpha$ be an integer with $0\le\alpha\le\lambda$. For each integer $i$ with $\alpha+1\le i\le\lambda$, fix $t_{i,1}$ with $0\le t_{i,1}\le t_i$. Then the following is true.

			\begin{enumerate}[label=(\alph*)]
				\item If $\alpha=0$, then
				{\allowdisplaybreaks
					\begin{align*}
						&d_{r_1,\dots,r_\lambda}^{t_{1,1},\dots,t_{\lambda,1}}\binom{\ell-\sum\limits_{1\le i\le \lambda} t_{i,1}r_i}{t_{1,2},\dots,t_{\lambda,2},\ell-\sum\limits_{1\le i\le \lambda} t_{i,1}r_i-\sum\limits_{1\le i\le \lambda} t_{i,2}}\\
						&\qquad=\enspace\left((-1)^{-1+\sum\limits_{1\le i\le\lambda} t_{i,1}}\right)\left( \frac{\ell}{\prod\limits_{1\le i\le\lambda} t_{i,1}! t_{i,2}! }\right)\sum_{0\le v_0\le -1+\sum\limits_{1\le i\le \lambda} t_i}  \bold{s}(-1+\sum\limits_{1\le i\le \lambda} t_i,v_0)\cdot\\
						&\pushright{(\ell-1-\sum\limits_{1\le i\le\ell} t_{i,1}(r_i-1))^{v_0}}.
				\end{align*}}%
				\item If $1\le \alpha< \lambda$, then
				{\allowdisplaybreaks
					\begin{align}\label{eq.uptoalpha}
						&\sum_{\substack{ 0\le t_{1,1}\le t_1,\\
								\dots\\
								0\le t_{\alpha,1}\le t_\alpha}} d_{r_1,\dots,r_\lambda}^{t_{1,1},\dots,t_{\lambda,1}}\binom{\ell-\sum\limits_{1\le i\le \lambda} t_{i,1}r_i}{t_{1,2},\dots,t_{\lambda,2},\ell-\sum\limits_{1\le i\le \lambda} t_{i,1}r_i-\sum\limits_{1\le i\le \lambda} t_{i,2}}\\
						&\qquad=\enspace\left((-1)^{-1+\sum\limits_{\alpha<i\le\lambda} t_{i,1}}\right)\left(\frac{\ell}{\prod\limits_{\alpha< i\le\lambda} t_{i,1}! t_{i,2}! }\right)\sum_{\substack{0\le v_0\le -1+\sum\limits_{1\le i\le \lambda} t_i\\ 0\le v_j\le v_0-\sum\limits_{1\le k<j} v_k\\ \text{for } 1\le j\le\alpha}}\bold{s}(-1+\sum\limits_{1\le i\le \lambda} t_i,v_0)\cdot\nonumber\\
						&\pushright{(\ell-1-\sum\limits_{\alpha<i\le \lambda} t_{i,1}(r_i-1))^{v_0-\sum\limits_{1\le i\le\alpha} v_i}\prod_{1\le i\le \alpha} (-1)^{t_i}(1-r_i)^{v_i}\bold{S}(v_i,t_i)\binom{v_0-\sum\limits_{1\le j< i} v_j}{v_i}.}\nonumber
				\end{align}}%
				\item 
				If $\alpha=\lambda$, then		
				\begin{equation}\label{eq.lambda}
					\begin{split}
						&\sum_{\substack{ 0\le t_{1,1}\le t_1\\
								\cdot\cdot\cdot\\
								0\le t_{\lambda,1}\le t_\lambda}} d_{r_1,\dots,r_\lambda}^{t_{1,1},\dots,t_{\lambda,1}}\binom{\ell-\sum\limits_{1\le i\le \lambda} t_{i,1}r_i}{t_{1,2},\dots,t_{\lambda,2},\ell-\sum\limits_{1\le i\le \lambda} t_{i,1}r_i-\sum\limits_{1\le i\le \lambda} t_{i,2}}\\
						&=-\ell\sum_{\substack{0\le v_0\le -1+\sum\limits_{1\le i\le \lambda} t_i\\ 0\le v_j\le v_0-\sum\limits_{1\le k<j} v_k\\ \text{for } 1\le j\le\lambda}}\left(\bold{s}(-1+\sum\limits_{1\le i\le \lambda} t_i,v_0)(\ell-1)^{v_0-\sum\limits_{1\le i\le\lambda} v_i}\cdot\right. \\
						&\pushright{\left.\prod_{1\le i\le \lambda} (-1)^{t_i}(1-r_i)^{v_i}\bold{S}(v_i,t_i)\binom{v_0-\sum\limits_{1\le j< i} v_j}{v_i}\right).}
					\end{split}
				\end{equation}

			\end{enumerate}

		\end{claim}
		
		\begin{proof}[Proof of Claim \ref{cl.recur}]

			(a) Assume that $\alpha=0$. Then
			{\allowdisplaybreaks
				\begin{align*}
					& d_{r_1,\dots,r_\lambda}^{t_{1,1},\dots,t_{\lambda,1}}\binom{\ell-\sum\limits_{1\le i\le \lambda} t_{i,1}r_i}{t_{1,2},\dots,t_{\lambda,2},\ell-\sum\limits_{1\le i\le \lambda} t_{i,1}r_i-\sum\limits_{1\le i\le \lambda} t_{i,2}}\\
					&\qquad=\enspace\left((-1)^{-1+\sum\limits_{1\le i\le\lambda} t_{i,1}}\right)\left(\frac{(-1+\sum\limits_{1\le i\le\lambda} t_{i,1})!}{\prod\limits_{1\le i\le\lambda} t_{i,1}! }\right)\ell\binom{\ell-1-\sum\limits_{1\le i\le \lambda} t_{i,1}r_i+\sum\limits_{1\le i\le \lambda} t_{i,1}}{-1+\sum\limits_{1\le i\le \lambda} t_{i,1}}\cdot\\
					&\pushright{\binom{\ell-\sum\limits_{1\le i\le \lambda} t_{i,1}r_i}{t_{1,2},\dots,t_{\lambda,2},\ell-\sum\limits_{1\le i\le \lambda} t_{i,1}r_i-\sum\limits_{1\le i\le \lambda} t_{i,2}}}\\
					&\qquad=\enspace\left((-1)^{-1+\sum\limits_{1\le i\le\lambda} t_{i,1}}\right)\left(\frac{1}{\prod\limits_{1\le i\le\lambda} t_{i,1}! t_{i,2}! }\right)\ell\prod_{0\le j\le -2+\sum\limits_{1\le i\le \lambda} t_i} (\ell-1-\sum_{1\le i\le\lambda} t_{i,1}(r_i-1)-j)\\
					&\qquad\stackrel{\text{use }\eqref{eq.stirling1}}{=}\enspace\left((-1)^{-1+\sum\limits_{1\le i\le\lambda} t_{i,1}}\right)\left(\frac{1}{\prod\limits_{1\le i\le\lambda} t_{i,1}! t_{i,2}! }\right)\ell\sum_{0\le v_0\le -1+\sum\limits_{1\le i\le \lambda} t_i} \bold{s}(-1+\sum_{1\le i\le\lambda} t_i,v_0)\cdot\\
					&\pushright{(\ell-1-\sum_{1\le i\le\lambda} t_{i,1}(r_i-1))^{v_0}}.
			\end{align*}}%

			(b) We use induction on $\alpha$. Let $\alpha=1$. Then
			
			{\allowdisplaybreaks
				\begin{align*}
					&\sum_{\substack{ 0\le t_{1,1}\le t_1}} d_{r_1,\dots,r_\lambda}^{t_{1,1},\dots,t_{\lambda,1}}\binom{\ell-\sum\limits_{1\le i\le \lambda} t_{i,1}r_i}{t_{1,2},\dots,t_{\lambda,2},\ell-\sum\limits_{1\le i\le \lambda} t_{i,1}r_i-\sum\limits_{1\le i\le \lambda} t_{i,2}}\\
					&\qquad\stackrel{(a)}{=}\enspace\sum_{0\le t_{1,1}\le t_1}\left((-1)^{-1+\sum\limits_{1\le i\le\lambda} t_{i,1}}\right)\left( \frac{\ell}{\prod\limits_{1\le i\le\lambda} t_{i,1}! t_{i,2}! }\right)\sum_{0\le v_0\le -1+\sum\limits_{1\le i\le \lambda} t_i}  \bold{s}(-1+\sum\limits_{1\le i\le \lambda} t_i,v_0)\cdot\\
					&\pushright{(\ell-1-\sum\limits_{1\le i\le\ell} t_{i,1}(r_i-1))^{v_0}}	\\
					&\qquad=\enspace\sum_{0\le t_{1,1}\le t_1}\left((-1)^{-1+\sum\limits_{1\le i\le\lambda} t_{i,1}}\right)\left(\frac{\ell}{\prod\limits_{1\le i\le\lambda} t_{i,1}! t_{i,2}! }\right)\sum_{0\le v_0\le -1+\sum\limits_{1\le i\le \lambda} t_i}  \bold{s}(-1+\sum\limits_{1\le i\le \lambda} t_i,v_0)\cdot\\
					&\pushright{\sum_{0\le v_1\le v_0}\binom{v_0}{v_1}(\ell-1-\sum_{1<i\le\lambda} t_{i,1}(r_i-1))^{v_0-v_1}(t_{1,1}(1-r_1))^{v_1}}\\
					&\qquad=\enspace\left((-1)^{-1+\sum\limits_{1<i\le\lambda} t_{i,1}}\right)\left(\frac{\ell}{\prod\limits_{1< i\le\lambda} t_{i,1}! t_{i,2}! }\right)\sum_{\substack{0\le v_0\le -1+\sum\limits_{1\le i\le \lambda} t_i\\ 0\le v_1\le v_0}}\bold{s}(-1+\sum\limits_{1\le i\le \lambda} t_i,v_0)\cdot\\
					&\pushright{(\ell-1-\sum\limits_{1<i\le \lambda} t_{i,1}(r_i-1))^{v_0-v_1}(1-r_1)^{v_1}\binom{v_0}{v_1}\frac{1}{t_1 !}\sum_{0\le t_{1,1}\le t_1} (-1)^{t_{1,1}}\binom{t_1}{t_{1,1}}t_{1,1}^{v_1}}\\
					&\qquad\stackrel{\text{use }\eqref{eq.stirling2}}{=}\enspace\left((-1)^{-1+\sum\limits_{1<i\le\lambda} t_{i,1}}\right)\left(\frac{\ell}{\prod\limits_{1< i\le\lambda} t_{i,1}! t_{i,2}! }\right)\sum_{\substack{0\le v_0\le -1+\sum\limits_{1\le i\le \lambda} t_i\\ 0\le v_1\le v_0}}\bold{s}(-1+\sum\limits_{1\le i\le \lambda} t_i,v_0)\cdot\\
					&\pushright{(\ell-1-\sum\limits_{1<i\le \lambda} t_{i,1}(r_i-1))^{v_0-v_1}(1-r_1)^{v_1}\binom{v_0}{v_1}(-1)^{t_1} \bold{S}(v_1,t_1)}\\
					&\qquad=\enspace\left((-1)^{-1+\sum\limits_{1<i\le\lambda} t_{i,1}}\right)\left(\frac{\ell}{\prod\limits_{1< i\le\lambda} t_{i,1}! t_{i,2}! }\right)\sum_{\substack{0\le v_0\le -1+\sum\limits_{1\le i\le \lambda} t_i\\ 0\le v_1\le v_0}}\bold{s}(-1+\sum\limits_{1\le i\le \lambda} t_i,v_0)\cdot\\
					&\pushright{(\ell-1-\sum\limits_{1<i\le \lambda} t_{i,1}(r_i-1))^{v_0-v_1}(-1)^{t_1}(1-r_1)^{v_1}\bold{S}(v_1,t_1)\binom{v_0}{v_1} }.
			\end{align*}}%
			So the base case $\alpha=1$ is true. Assume that Claim \ref{cl.recur} is true up to $\alpha-1\le \lambda-2$. Then
			{\allowdisplaybreaks
				\begin{align*}
					&\sum_{\substack{ 0\le t_{1,1}\le t_1,\\
							\dots\\
							0\le t_{\alpha,1}\le t_\alpha}} d_{r_1,\dots,r_\lambda}^{t_{1,1},\dots,t_{\lambda,1}}\binom{\ell-\sum\limits_{1\le i\le \lambda} t_{i,1}r_i}{t_{1,2},\dots,t_{\lambda,2},\ell-\sum\limits_{1\le i\le \lambda} t_{i,1}r_i-\sum\limits_{1\le i\le \lambda} t_{i,2}}\\
					&\qquad=\enspace\sum_{0\le t_{\alpha,1}\le t_\alpha}\left(\sum_{\substack{ 0\le t_{1,1}\le t_1,\\
							\dots\\
							0\le t_{\alpha-1,1}\le t_{\alpha-1}}} d_{r_1,\dots,r_\lambda}^{t_{1,1},\dots,t_{\lambda,1}}\binom{\ell-\sum\limits_{1\le i\le \lambda} t_{i,1}r_i}{t_{1,2},\dots,t_{\lambda,2},\ell-\sum\limits_{1\le i\le \lambda} t_{i,1}r_i-\sum\limits_{1\le i\le \lambda} t_{i,2}}\right)	\\
					&\qquad\stackrel{(*)}{=}\enspace\sum_{0\le t_{\alpha,1}\le t_\alpha}\left((-1)^{-1+\sum\limits_{\alpha-1<i\le\lambda} t_{i,1}}\right)\left(\frac{\ell}{\prod\limits_{\alpha-1< i\le\lambda} t_{i,1}! t_{i,2}! }\right)\sum_{\substack{0\le v_0\le -1+\sum\limits_{1\le i\le \lambda} t_i\\ 0\le v_j\le v_0-\sum\limits_{1\le k<j} v_k\\ \text{for } 1\le j\le\alpha-1}}\bold{s}(-1+\sum\limits_{1\le i\le \lambda} t_i,v_0)\cdot\\
					&\pushright{(\ell-1-\sum\limits_{\alpha-1<i\le \lambda} t_{i,1}(r_i-1))^{v_0-\sum\limits_{1\le i\le\alpha-1} v_i}\prod_{1\le i\le \alpha-1} (-1)^{t_i}(1-r_i)^{v_i}\bold{S}(v_i,t_i)\binom{v_0-\sum\limits_{1\le j< i} v_j}{v_i}}\\
					&\qquad=\enspace\sum_{0\le t_{\alpha,1}\le t_\alpha}\left((-1)^{-1+\sum\limits_{\alpha-1<i\le\lambda} t_{i,1}}\right)\left(\frac{\ell}{\prod\limits_{\alpha-1< i\le\lambda} t_{i,1}! t_{i,2}! }\right)\sum_{\substack{0\le v_0\le -1+\sum\limits_{1\le i\le \lambda} t_i\\ 0\le v_j\le v_0-\sum\limits_{1\le k<j} v_k\\ \text{for } 1\le j\le\alpha-1}}\bold{s}(-1+\sum\limits_{1\le i\le \lambda} t_i,v_0)\cdot\\
					&\pushright{\sum_{0\le v_\alpha\le v_0-\sum\limits_{1\le i\le\alpha-1} v_i}\binom{v_0-\sum\limits_{1\le i\le\alpha-1} v_i}{v_\alpha}(\ell-1-\sum_{\alpha<i\le\lambda} t_{i,1}(r_i-1))^{v_0-\sum\limits_{1\le i\le\alpha} v_i}\cdot}\\
					&\pushright{(t_{\alpha,1}(1-r_\alpha))^{v_\alpha}\prod_{1\le i\le \alpha-1} (-1)^{t_i}(1-r_i)^{v_i}\bold{S}(v_i,t_i)\binom{v_0-\sum\limits_{1\le j< i} v_j}{v_i}}\\
					&\qquad=\enspace\left((-1)^{-1+\sum\limits_{\alpha<i\le\lambda} t_{i,1}}\right)\left(\frac{\ell}{\prod\limits_{\alpha< i\le\lambda} t_{i,1}! t_{i,2}! }\right)\sum_{\substack{0\le v_0\le -1+\sum\limits_{1\le i\le \lambda} t_i\\ 0\le v_j\le v_0-\sum\limits_{1\le k<j} v_k\\ \text{for } 1\le j\le\alpha}}\bold{s}(-1+\sum\limits_{1\le i\le \lambda} t_i,v_0)\cdot\\
					&\pushright{(\ell-1-\sum\limits_{\alpha<i\le \lambda} t_{i,1}(r_i-1))^{v_0-\sum\limits_{1\le i\le\alpha} v_i}\prod_{1\le i\le \alpha-1} (-1)^{t_i}(1-r_i)^{v_i}\bold{S}(v_i,t_i)\binom{v_0-\sum\limits_{1\le j< i} v_j}{v_i}\cdot}\\
					&\pushright{(1-r_\alpha)^{v_\alpha}\binom{v_0-\sum\limits_{1\le i\le\alpha-1} v_i}{v_\alpha}\frac{1}{t_\alpha !}\sum_{0\le t_{\alpha,1}\le t_\alpha} (-1)^{t_{\alpha,1}}\binom{t_\alpha}{t_{\alpha,1}}t_{\alpha,1}^{v_\alpha}}\\
					&\qquad\stackrel{\text{use }\eqref{eq.stirling2}}{=}\enspace\left((-1)^{-1+\sum\limits_{\alpha<i\le\lambda} t_{i,1}}\right)\left(\frac{\ell}{\prod\limits_{\alpha< i\le\lambda} t_{i,1}! t_{i,2}! }\right)\sum_{\substack{0\le v_0\le -1+\sum\limits_{1\le i\le \lambda} t_i\\ 0\le v_j\le v_0-\sum\limits_{1\le k<j} v_k\\ \text{for } 1\le j\le\alpha}}\bold{s}(-1+\sum\limits_{1\le i\le \lambda} t_i,v_0)\cdot\\
					&\pushright{(\ell-1-\sum\limits_{\alpha<i\le \lambda} t_{i,1}(r_i-1))^{v_0-\sum\limits_{1\le i\le\alpha} v_i}\prod_{1\le i\le \alpha-1} (-1)^{t_i}(1-r_i)^{v_i}\bold{S}(v_i,t_i)\binom{v_0-\sum\limits_{1\le j< i} v_j}{v_i}\cdot}\\
					&\pushright{(1-r_\alpha)^{v_\alpha}\binom{v_0-\sum\limits_{1\le i\le \alpha-1} v_i}{v_\alpha}(-1)^{t_\alpha} \bold{S}(v_\alpha,t_\alpha)}\\
					&\qquad=\enspace\left((-1)^{-1+\sum\limits_{\alpha<i\le\lambda} t_{i,1}}\right)\left(\frac{\ell}{\prod\limits_{\alpha< i\le\lambda} t_{i,1}! t_{i,2}! }\right)\sum_{\substack{0\le v_0\le -1+\sum\limits_{1\le i\le \lambda} t_i\\ 0\le v_j\le v_0-\sum\limits_{1\le k<j} v_k\\ \text{for } 1\le j\le\alpha}}\bold{s}(-1+\sum\limits_{1\le i\le \lambda} t_i,v_0)\cdot\\
					&\pushright{(\ell-1-\sum\limits_{\alpha<i\le \lambda} t_{i,1}(r_i-1))^{v_0-\sum\limits_{1\le i\le\alpha} v_i}\prod_{1\le i\le \alpha} (-1)^{t_i}(1-r_i)^{v_i}\bold{S}(v_i,t_i)\binom{v_0-\sum\limits_{1\le j< i} v_j}{v_i}.}
			\end{align*}}%
			Note that $(*)$ is by the induction hypothesis of the Claim \ref{cl.recur} for $\alpha-1$.	Hence part (b) of Claim \ref{cl.recur} is true. 
			
			(c) Formula \eqref{eq.lambda} comes from Formula \eqref{eq.uptoalpha} by taking $\alpha=\lambda$ and letting the empty sums $\sum\limits_{\alpha<i\le\lambda} t_{i,1}$ and $\sum\limits_{\alpha<i\le \lambda} t_{i,1}(r_i-1)$ to be $0$, and the empty product $\prod\limits_{\alpha< i\le\lambda} t_{i,1}! t_{i,2}!$ to be $1$. The proof in part (b) still works when $\alpha=\lambda$. This completes the proof of Claim \ref{cl.recur}.
		\end{proof}

		\noindent{\it Proof of Theorem \ref{thm.form} \ref{it.el_m} (continued).} From part (c) of Claim \ref{cl.recur}, we get that 
		{\allowdisplaybreaks	
			\begin{align}
				&\sum_{\substack{ 0\le t_{1,1}\le t_1\\
						\cdot\cdot\cdot\\
						0\le t_{\lambda,1}\le t_\lambda}} d_{r_1,\dots,r_\lambda}^{t_{1,1},\dots,t_{\lambda,1}}\binom{\ell-\sum\limits_{1\le i\le \lambda} t_{i,1}r_i}{t_{1,2},\dots,t_{\lambda,2},\ell-\sum\limits_{1\le i\le \lambda} t_{i,1}r_i-\sum\limits_{1\le i\le \lambda} t_{i,2}}\nonumber\\
				&=\enspace-\ell\sum_{\substack{0\le v_0\le -1+\sum\limits_{1\le i\le \lambda} t_i\\ 0\le v_j\le v_0-\sum\limits_{1\le k<j} v_k\\ \text{for } 1\le j\le\lambda}}\left(\bold{s}(-1+\sum\limits_{1\le i\le \lambda} t_i,v_0)(\ell-1)^{v_0-\sum\limits_{1\le i\le\lambda} v_i}\prod_{1\le i\le \lambda} (-1)^{t_i}(1-r_i)^{v_i}\bold{S}(v_i,t_i)\binom{v_0-\sum\limits_{1\le j< i} v_j}{v_i}\right)\label{eq.sum}\\
				&=\enspace 0\nonumber.
		\end{align}}%
		The last equality holds because each term in the sum of line \eqref{eq.sum} is $0$. Indeed, in the factor $\prod_{1\le i\le \lambda} (-1)^{t_i}(1-r_i)^{v_i}\bold{S}(v_i,t_i)\binom{v_0-\sum\limits_{1\le j< i} v_j}{v_i}$, there exists $1\le i\le \lambda$ such that $v_i<t_i$ and hence $\bold{S}(v_i,t_i)=0$. We can show this by contradiction. If $v_i\ge t_i$ for each $1\le i\le \lambda$, then $v_0-\sum\limits_{0\le i\le\lambda} v_i\le v_0-\sum\limits_{0\le i\le\lambda} t_i\le -1$. But $v_\lambda\le v_0-\sum\limits_{1\le i<\lambda} v_i$ and therefore $v_0-\sum\limits_{0\le i\le\lambda} v_i\ge 0$. This is a contradiction.

		Note that if $(t_{1,1},\dots,t_{\lambda,1})= (t_1,\dots,t_\lambda)$, then $$\binom{\ell-\sum\limits_{1\le i\le \lambda} t_{i,1}r_i}{t_{1,2},\dots,t_{\lambda,2},\ell-\sum\limits_{1\le i\le \lambda} t_{i,1}r_i-\sum\limits_{1\le i\le \lambda} t_{i,2}}=\binom{\ell-\sum\limits_{1\le i\le \lambda} t_{i,1}r_i}{0,\dots,0,\ell-\sum\limits_{1\le i\le \lambda} t_{i,1}r_i}=1.$$ 
		Hence,
		\begin{equation*}
			d_{r_1,\dots,r_\lambda}^{t_1,\dots,t_\lambda}=-\sum_{\substack{ 0\le t_{1,1}\le t_1\\
					\cdot\cdot\cdot\\
					0\le t_{\lambda,1}\le t_\lambda\\
					(t_{1,1},\dots,t_{\lambda,1})\ne \\
					(t_1,\dots,t_\lambda)}} d_{r_1,\dots,r_\lambda}^{t_{1,1},\dots,t_{\lambda,1}}\binom{\ell-\sum\limits_{1\le i\le \lambda} t_{i,1}r_i}{t_{1,2},\dots,t_{\lambda,2},\ell-\sum\limits_{1\le i\le \lambda} t_{i,1}r_i-\sum\limits_{1\le i\le \lambda} t_{i,2}}.
		\end{equation*}
		Therefore, Equation \eqref{eq.coeff} holds. This completes the proof of part \ref{it.el_m} of Theorem \ref{thm.form}. 
		
		From \eqref{eq.coeff}, we know that $d_{r_1,\dots,r_\lambda}^{t_1,\dots,t_\lambda}$ can be represented in terms of $d_{r_1,\dots,r_\lambda}^{t_{1,1},\dots,t_{\lambda,1}}$ with $0\le t_{i,1}\le t_i$, $0\le \sum r_i t_{i,1}<\sum r_i t_i$, and the binomial coefficients using product, summation, and negation. Because $d_{r_1,\dots,r_\lambda}^{0,\dots,0}=-1$ is an integer, we get that $d_{r_1,\dots,r_\lambda}^{t_1,\dots,t_\lambda}$ is an integer. This proves Theorem \ref{thm.form} \ref{it.int} when $0<m<\ell$.
		
		\vspace{0.2cm}
		
		{\it Proof of Theorem \ref{thm.form} \ref{it.el_0}}	We start with the following claim.
		\begin{claim}\label{cl.coeff5}
			The coefficient of the term $q^{-\ell^2}$ in the $q$-expansion of the left-hand side of \eqref{eq.construction}, i.e. $j(\ell z)^{\ell+1}+j(z)^{\ell+1}+\sum\limits_{0\le r,s\le \ell}a_{r,s}j(\ell z)^r j(z)^s$, is 
			\begin{equation}\label{eq.coeff4}
				\aligned
				a_{\ell,0}+&a_{\ell-1,\ell}+a_{\ell,\ell}\ell c_0+(\ell+1)c_0+\\
				&\sum_{0\le n\le \ell-1} a_{\ell,\ell-n}\left(\sum\limits_{\substack{t_1+t_2 2+\cdots+t_{\ell-n}(\ell-n)=\ell-n\\
						t_i\ge 0}}\binom{\ell-n}{t_1,\dots,t_{\ell-n},\ell-n-\sum\limits_{1\le i\le \ell-n} t_i}c_0^{t_1}\cdots c_{\ell-n-1}^{t_{\ell-n}}\right).
				\endaligned
			\end{equation}
		\end{claim}

		Assume that $0\le r\le \ell$ and $0\le s\le\ell+1$ satisfy one of the following three conditions: (a) $r=0$ and $s=\ell+1$; (b) $0\le r\le\ell-2$ and $0\le s\le \ell$; (c) $r=\ell-1$ and $0\le s\le \ell-1$. Then the lowest degree in the $q$-expansion of $a_{r,s}j(\ell z)^rj(z)^s$ is greater than $-\ell^2$.

		The coefficients of the term $q^{-\ell^2}$ in the $q$-expansion of $j(\ell z)^{\ell}$ and $j(\ell z)^{\ell-1}j(z)^\ell$ are both $1$. The coefficient of the term $q^{-\ell^2}$ in the $q$-expansion of $j(\ell z)^{\ell+1}$ is $(\ell+1)c_0$. This is clear because $j(\ell z)^{\ell+1}=q^{-\ell(\ell+1)}(\sum\limits_{i=0}^\infty c_{i-1} q^{\ell i})^{\ell+1}$.

		\begin{claim}\label{cl.A_n_l_0}
			Let $1\le n\le \ell-n$ and let $A_n$ be the coefficient of the term $q^{-\ell^2}$ in the $q$-expansion of $j(\ell z)^{\ell}j(z)^{\ell-n}$. Then the following is true.
			\begin{enumerate}[label=(\alph*)]
				\item\label{it.A_0} $$A_0=\ell c_0+\sum\limits_{\substack{t_1+t_2 2+\cdots+t_{\ell}\ell=\ell\\
						t_i\ge 0}}\binom{\ell}{t_1,\dots,t_{\ell},\ell-\sum\limits_{1\le i\le \ell} t_i}c_0^{t_1}\cdots c_{\ell-1}^{t_{\ell}}.$$
				\item\label{it.A_n} If $1\le n\le\ell-1$, then $$A_n=\sum\limits_{\substack{t_1+t_2 2+\cdots+t_{\ell-n}(\ell-n)=\ell-n\\
						t_i\ge 0}}\binom{\ell-n}{t_1,\dots,t_{\ell-n},\ell-n-\sum\limits_{1\le i\le \ell-n} t_i}c_0^{t_1}\cdots c_{\ell-n-1}^{t_{\ell-n}}.$$
			\end{enumerate}
		\end{claim}

		Recall that
		$$j(\ell z)^\ell j(z)^{\ell}=q^{-\ell^2-\ell}(\sum\limits_{i=0}^\infty c_{i-1} q^{\ell i})^\ell(\sum\limits_{i=0}^\infty c_{i-1} q^i)^{\ell}.$$

		In the $q$-expansion of $(\sum\limits_{i=0}^\infty c_{i-1} q^{\ell i})^\ell$, the lowest degree term is $1$, the second lowest degree term is $\ell c_0 q^\ell$, and the third lowest degree is $2\ell$. So $A_0$ is the sum of $\ell c_0$ and the coefficient of $q^\ell$ in the $q$-expansion of $(\sum\limits_{i=0}^\infty c_{i-1} q^i)^{\ell}$.

		Assume that $1\le n\le\ell-1$. Then
		$$j(\ell z)^\ell j(z)^{\ell-n}=q^{-\ell^2-\ell+n}(\sum\limits_{i=0}^\infty c_{i-1} q^{\ell i})^\ell(\sum\limits_{i=0}^\infty c_{i-1} q^i)^{\ell-n}.$$

		In the $q$-expansion of $(\sum\limits_{i=0}^\infty c_{i-1} q^{\ell i})^\ell$, the lowest degree term is $1$, the second lowest degree is $\ell$. So $A_n$ is the coefficient of $q^{\ell-n}$ in the $q$-expansion of $(\sum\limits_{i=0}^\infty c_{i-1} q^i)^{\ell-n}$.
		By the Multinomial Theorem as used in the proof of Claim \ref{cl.A_n}, we know that part \ref{it.A_0} and \ref{it.A_n} of Claim \ref{cl.A_n_l_0} hold. 
		
		It follows that the coefficient of the term $q^{-\ell^2}$ in the $q$-expansion of $j(\ell z)^{\ell+1}+j(z)^{\ell+1}+\sum\limits_{0\le r,s\le \ell}a_{r,s}j(\ell z)^r j(z)^s$ is \eqref{eq.coeff4}. Therefore Claim \ref{cl.coeff5} is true.
		
		Since $\Phi_\ell(j(\ell z), j(z))=0$ in \eqref{eq.construction}, we find that the expression \eqref{eq.coeff4} equals $0$. Therefore,
		\begin{equation*}
			\aligned
			a_{\ell,0}=\enspace&
			-a_{\ell-1,\ell}-a_{\ell,\ell}\ell c_0-(\ell+1)c_0-\\
			&\sum_{0\le n\le \ell-1} a_{\ell,\ell-n}\left(\sum\limits_{\substack{t_1+t_2 2+\cdots+t_{\ell-n}(\ell-n)=\ell-n\\
					t_i\ge 0}}\binom{\ell-n}{t_1,\dots,t_{\ell-n},\ell-n-\sum\limits_{1\le i\le \ell-n} t_i}c_0^{t_1}\cdots c_{\ell-n-1}^{t_{\ell-n}}\right)\\
			=\enspace& -(\ell+1)c_0-\\
			&\sum_{0\le n\le \ell-1} a_{\ell,\ell-n}\left(\sum\limits_{\substack{t_1+t_2 2+\cdots+t_{\ell-n}(\ell-n)=\ell-n\\
					t_i\ge 0}}\binom{\ell-n}{t_1,\dots,t_{\ell-n},\ell-n-\sum\limits_{1\le i\le \ell-n} t_i}c_0^{t_1}\cdots c_{\ell-n-1}^{t_{\ell-n}}\right),
			\endaligned
		\end{equation*}
		where we used the facts that $a_{\ell,\ell}=-1$ and $a_{\ell-1,\ell}=a_{\ell,\ell-1}=\ell c_0$ proved in Theorem \ref{thm.form} \ref{it.el_m}.
		
		The rest of the proof of Theorem \ref{thm.form} \ref{it.el_0} is the same as the proof of Theorem \ref{thm.form} \ref{it.el_m}. We simply take $m=\ell$ after the line \eqref{eq.ind} in the proof of Theorem \ref{thm.form} \ref{it.el_m}. The case $m=\ell$ in Theorem \ref{thm.form} \ref{it.int} follows as in the case $0<m<\ell$. The proof of Theorem \ref{thm.form} is now complete.
	\end{proof}
\end{emp}

\begin{example}\label{eg.firstfew}
	For use later in Proposition \ref{prop.mle7}, we list explicitly the first few coefficients $a_{\ell,\ell-m}$ with $m<\ell$ in Theorem \ref{thm.form}: $a_{\ell,\ell-1}=\ell c_0$, $a_{\ell,\ell-2}=\ell c_1-\binom{\ell}{2}c_0^2$, $a_{\ell,\ell-3}=\ell c_2-\ell(\ell-2)c_0 c_1+\binom{\ell}{3}c_0^3$, 
	
	{\allowdisplaybreaks
		\begin{align*}
			a_{\ell,\ell-4}&=\ell c_3-\ell(\ell-3)(\frac{c_1^2}{2}+c_0c_2)+\ell\binom{\ell-2}{2}c_0^2c_1-\binom{\ell}{4}c_0^4,\\
			a_{\ell,\ell-5}&=\ell c_4-\ell(\ell-4)(c_0 c_3+c_1 c_2)+\ell\binom{\ell-3}{2}(c_0^2c_2+c_0 c_1^2)-\ell\binom{\ell-2}{3}c_0^3 c_1+\binom{\ell}{5}c_0^5,\\
			a_{\ell,\ell-6}&=\ell c_5-\ell(\ell-5)(c_4 c_0+c_3 c_1+\frac{1}{2}c_2^2)+\ell\binom{\ell-4}{2}(c_3 c_0^2+2c_2 c_1 c_0+\frac{1}{3}c_1^3)\\
			&\pushright{-\ell\binom{\ell-3}{3}(c_2 c_0^3+\frac{3}{2}c_1^2 c_0^2)+\ell\binom{\ell-2}{4}c_1 c_0^4-\binom{\ell}{6}c_0^6},\\
			a_{\ell,\ell-7}&=\ell c_6-\ell(\ell-6)(c_2c_3+c_0c_5+c_1c_4)+\ell\binom{\ell-5}{2}(c_0^2c_4+2c_0c_1c_3+c_1^2c_2+c_0c_2^2)\\
			&\pushright{-\ell\binom{\ell-4}{3}(3c_0^2c_1c_2+c_0^3c_3+c_0 c_1^3)+\ell\binom{\ell-3}{4}(c_0^4c_2+2c_0^3c_1^2)-\ell\binom{\ell-2}{5}c_0^5c_1+\binom{\ell}{7}c_0^7}.
	\end{align*}}%
\end{example}

We now use Theorem \ref{thm.form} to provide some evidence supporting Conjecture \ref{conj.div}. Recall that part (a) of Conjecture \ref{conj.div} implies that $a_{\ell,0}$ is divisible by $2^{15}$. In our next proposition, we prove divisibilities by powers of $2$ for more general coefficients $a_{\ell,\ell-m}$ with $m\le\ell$.

\begin{proposition}\label{prop.2Div}
	
	Let $\ell$ be an odd prime and let $0<m\le \ell$ be an integer. Let $a_{\ell,\ell-m}$ be the coefficient of the $\ell$-th classical modular polynomial $\Phi_\ell(X,Y)$. Then

	\begin{enumerate}[label=(\roman*)]
		\item If $m\equiv 4\operatorname{mod} 8$, then $a_{\ell,\ell-m}\equiv 0\operatorname{mod} 2$.
		\item If $m\equiv 2\operatorname{mod} 4$, then $a_{\ell,\ell-m}\equiv 0\operatorname{mod} 2^2$.
		\item If $m\equiv 1\operatorname{mod} 8$, then $a_{\ell,\ell-m}\equiv 0\operatorname{mod} 2^3$.
		\item If $m\equiv 5\operatorname{mod} 8$, then 
		$a_{\ell,\ell-m}\equiv 0\operatorname{mod} 2^4$.
		\item If $m\equiv 3\operatorname{mod} 4$, then $a_{\ell,\ell-m}\equiv 0\operatorname{mod} 2^5$. 
	\end{enumerate}

\end{proposition}

\begin{proof}
	First, we consider the case $m=\ell$ and prove the congruences for $a_{\ell,0}$ using Equation \eqref{eq.a_l_0}. Recall that $c_0=2^3\cdot 3\cdot 31$. So $(\ell+1)c_0 \equiv 0\operatorname{mod} 2^4$. Assume that $t_1 r_1+\cdots +t_\lambda r_\lambda=\ell$ for some $0<r_1<\cdots <r_\lambda$ and $t_1,\dots,t_\lambda\ge 1$.
	
	If $r_1\ge 2$, then $r_i\ge 3$ is odd for some $1\le i\le\lambda$. Hence $c_{r_i-1}\equiv 0\operatorname{mod} 2^{11}$ by Theorem 1 in \cite{Lehner}. If $r_1=1$ and $t_1\ge 2$, then $c_{r_1-1}^{t_1}=c_0^{t_1}\equiv 0\operatorname{mod} 2^{6}$. If $r_1=t_1=1$ and $r_i$ is odd for some $1\le i\le\lambda$. Then $c_{r_i-1}\equiv 0\operatorname{mod} 2^{11}$ by Theorem 1 in \cite{Lehner}. Assume that $r_1=t_1=1$ and $r_i$ is even for each $2\le i\le \lambda$.

	Assume that $\ell\equiv 1\operatorname{mod} 8$. then $c_{r_1-1}=c_0\equiv 0\operatorname{mod} 2^3$. So each summand on the right-hand side of \eqref{eq.a_l_0} is divisible by $2^3$ and hence, $2^3$ divides $a_{\ell,0}$.

	Assume that $\ell\equiv 5\operatorname{mod} 8$. 
	If $r_i\not\equiv 0\operatorname{mod} 8$ for some $2\le i\le\lambda$, then $c_{r_i-1}\equiv 0\operatorname{mod} 2$ by relations (2),(3) and (4) in \cite{Kolberg_2}. So $2^4$ divides $c_{r_1-1}^{t_1}\cdots c_{r_\lambda-1}^{t_\lambda}$. If $r_i\equiv 0\operatorname{mod} 8$ for each $2\le i\le\lambda$, then $\ell\equiv 1\operatorname{mod} 8$. This is a contradiction. So each summand on the right-hand side of \eqref{eq.a_l_0} is divisible by $2^4$ and hence, $2^4$ divides $a_{\ell,0}$.

	Assume that $\ell\equiv 3\operatorname{mod} 4$. If $r_i\equiv 2\operatorname{mod} 4$ for some $2\le i\le\lambda$, then $c_{r_i-1}\equiv 0\operatorname{mod} 2^2$ by the relations (2) and (4) in \cite{Kolberg_2}. So $2^5$ divides $c_{r_1-1}^{t_1}\cdots c_{r_\lambda-1}^{t_\lambda}$. If $r_i\equiv 0\operatorname{mod} 8$ or $r_i\equiv 4\operatorname{mod} 8$ for each $2\le i\le\lambda$, then $\ell\equiv 1\operatorname{mod} 8$ or $\ell\equiv 5\operatorname{mod} 8$. This is a contradiction. So each summand on the right-hand side of \eqref{eq.a_l_0} is divisible by $2^5$ and hence, $2^5$ divides $a_{\ell,0}$. 
	
	The cases where $m<\ell$ are proved using the formula \eqref{eq.a_l_l_m} for $a_{\ell,\ell-m}$ in Theorem \ref{thm.form}. We leave the details to the reader.
\end{proof}

Recall that part (b) of Conjecture \ref{conj.div} implies that $a_{\ell,0}\equiv 0\operatorname{mod} 3^3$, and when $\ell\equiv 1\operatorname{mod} 3$, $a_{\ell,0}\equiv 0\operatorname{mod} 3^5$. Using Theorem \ref{thm.form}, we can show the following:

\begin{proposition}\label{prop.3Div}
	Let $\ell$ be a prime and let $0<m\le \ell$ be an integer. Let $a_{\ell,\ell-m}$ be the coefficient of the $\ell$-th classical modular polynomial $\Phi_\ell(X,Y)$. Then
	\begin{enumerate}[label=(\roman*)]
		\item If $m\equiv 1\operatorname{mod} 3$, then $a_{\ell,\ell-m}\equiv 0\operatorname{mod} 3$.
		\item If $m\equiv 2\operatorname{mod} 3$, then 
		$a_{\ell,\ell-m}\equiv 0\operatorname{mod} 3^2$.
	\end{enumerate}
\end{proposition}

\begin{proof}

	We first prove the congruence for $a_{\ell,0}$ using Equation \eqref{eq.a_l_0}. Recall that $c_0=2^3\cdot 3\cdot 31$. Assume that $t_1 r_1+\cdots +t_\lambda r_\lambda=\ell$ for some $0<r_1<\cdots <r_\lambda$ and $t_1,\dots,t_\lambda\ge 1$.

	Then $r_i\not\equiv 0\operatorname{mod} 3$ for some $1\le i\le \lambda$. If $r_i=1$, then $c_{r_i-1}=c_0\equiv 0\operatorname{mod} 3$. Assume that $r_i> 1$. If $r_i\equiv 1\operatorname{mod} 3$, then $c_{r_i-1}\equiv 0\operatorname{mod} 3^{5}$ by Theorem 1 in \cite{Lehner}. If $r_i\equiv 2\operatorname{mod} 3$, then $c_{r_i-1}\equiv 0\operatorname{mod} 3^{3}$ by 1.13 in \cite{Kolberg}. Assume that $r_1=1$ and $t_1\ge 2$. Then $c_0^{t_1}\equiv 0\operatorname{mod} 3^2$.

	Assume that $\ell\equiv 1\operatorname{mod} 3$. Then $(\ell+1)c_0 \equiv 0\operatorname{mod} 3$. So each summand on the right-hand side of \eqref{eq.a_l_0} is divisible by $3$ and hence, $3$ divides $a_{\ell,0}$.

	Assume that $\ell\equiv 2\operatorname{mod} 3$. Then $(\ell+1)c_0 \equiv 0\operatorname{mod} 3^2$. If $r_1=t_1=1$, then $r_2 t_2+\cdots+r_\lambda t_\lambda=\ell-1$. If $r_i\equiv 0\operatorname{mod} 3$ for each $2\le i\le \lambda$, then $\ell\equiv 1\operatorname{mod} 3$. This is a contradiction. So each summand on the right-hand side of \eqref{eq.a_l_0} is divisible by $3^2$ and hence, $3^2$ divides $a_{\ell,0}$. 
	
	The cases when $m<\ell$ are proved using the formula \eqref{eq.a_l_l_m} for $a_{\ell,\ell-m}$. We leave the details to the reader.
\end{proof}

Recall that Conjecture \ref{conj.div} part (c) predicts certain congruences modulo $5$ satisfied by $a_{m,n}$ with $\ell+1>m+n$. Our computations also suggest that the following $5$-divisibilities hold in some cases where $\ell+1\le m+n$.

\begin{conjecture}\label{conj.a_l_m_5Div}
	Let $\ell$ be a prime and let $0<m< \ell$ be an integer. Let $a_{\ell,\ell-m}$ be the coefficient of the $\ell$-th classical modular polynomial $\Phi_\ell(X,Y)$.
	Assume that one of the following three conditions is true:
	\begin{enumerate}[label=(\roman*)]
		\item $\ell\equiv 1\text{ or }3\operatorname{mod} 5$ and $m\equiv 4\operatorname{mod} 5$;
		\item $\ell\equiv 2\operatorname{mod} 5$ and $m\equiv 3\operatorname{mod} 5$;
		\item $\ell\equiv 4\operatorname{mod} 5$ and $m\equiv 2\operatorname{mod} 5$.
	\end{enumerate}
	Then $5$ divides $a_{\ell,\ell-m}$.
	
\end{conjecture}

Recall that $c_0=744$, $c_1=196884$, $c_2=21493760$, $c_3=864299970$, $c_4=20245856256$, $c_5=333202640600$, and $c_6=4252023300096$ (see e.g. \cite{vanWij} p. 398).

\begin{proposition}\label{prop.mle7}
	Conjecture \ref{conj.a_l_m_5Div} is true for $m\le 7$.
\end{proposition}

\begin{proof}

	We reduce the expressions in \ref{eg.firstfew} modulo $5$ and find: If $\ell\equiv 4\operatorname{mod} 5$, then $$a_{\ell,\ell-2}\equiv\ell c_1-\frac{\ell(\ell-1)}{2} c_0^2\equiv 0\operatorname{mod} 5,$$ and 
	
	\begin{align*}
		a_{\ell,\ell-7}&\equiv\ell c_6-\ell(\ell-6)c_1c_4+\ell\binom{\ell-5}{2}c_0^2c_4
		-\ell\binom{\ell-2}{5}c_0^5c_1+\binom{\ell}{7}c_0^7\equiv 0\operatorname{mod} 5.
	\end{align*}
	
	If $\ell\equiv 2\operatorname{mod} 5$, then  $$a_{\ell,\ell-3}\equiv-\ell(\ell-2)c_0 c_1+\binom{\ell}{3}c_0^3\equiv 0\operatorname{mod} 5.$$

	If $\ell\equiv 1\text{ or }3\operatorname{mod} 5$, then 	$$a_{\ell,\ell-4}\equiv-\frac{\ell(\ell-3)}{2}c_1^2+\ell\binom{\ell-2}{2}c_1c_0^2-\binom{\ell}{4}c_0^4\equiv 0\operatorname{mod} 5.$$
\end{proof}

Let $\ell$ be an odd prime such that $\legendre{-3}{\ell}=1$. The following theorem of Ito shows that Conjecture \ref{conj.div} is true for both $a_{0,0}$ and $a_{1,0}$ of $\Phi_{\ell}(X,Y)$ for such $\ell$.
\begin{theorem}\label{thm.ito}[Ito \cite{Ito}, see Theorem 1]
	Let $\ell\ge 5$ be a prime and let $a_{0,0}$ and $a_{1,0}$ be the coefficients of the $\ell$-th classical modular polynomial $\Phi_\ell(X,Y)$. Then $\ell\equiv 1\operatorname{mod} 6$ if and only if $a_{0,0}=a_{1,0}=0$.
\end{theorem}

For a general prime $\ell$, we prove the following results on the divisibilities of $a_{0,0}$ of $\Phi_\ell(X,Y)$.

\begin{proposition}\label{prop.a_0_0}
	Let $K$ be an imaginary quadratic field with ring of integers $\mathcal{O}_K$ and Hilbert class field $H$. Let $j$ be the $j$-invariant of an elliptic curve $E/H$ with geometric endomorphism ring $\mathcal{O}_K$. Assume that a prime $\ell\in\mathbb{N}$ splits completely in $H$. Let $a_{0,0}$ be the coefficient of the $\ell$-th classical modular polynomial $\Phi_\ell(X,Y)$. Then $|\operatorname{Norm}_\mathbb{Q}^{\mathbb{Q}(j)} j|$ divides $a_{0,0}$.
\end{proposition}

\begin{proof}
	Because $\ell$ splits completely in $H$, we have that $\ell$ splits completely as $\mathfrak{p}_1\mathfrak{p}_2$ in $K$ where $\mathfrak{p}_1$ and $\mathfrak{p}_1$ are principal ideals. Assume that $\mathfrak{p}_1=(\alpha)$ for some $\alpha\in\mathcal{O}_K$. Then $|\operatorname{Norm}_\mathbb{Q}^K \alpha|=\ell$. Let $[\alpha]: E\rightarrow E$ be the endomorphism as defined in Proposition II 1.1 of \cite{SilvermanAd}. Then $[\alpha]$ has degree $\ell$ by Corollary II 1.5 in \cite{SilvermanAd}. It follows that $\Phi_\ell(j,j)=0$ and hence $j$ is a root of the polynomial $P(X):=\Phi_\ell(X,X)$. Note that $P(X)$ has leading term $-X^{2\ell}$. In particular, $P(X)$ is primitive. 
	
	Let $f(X)\in\mathbb{Z}[X]$ be a primitive minimal polynomial of $j$ over $\mathbb{Q}$. Then $P(X)=f(X)g(X)$ for some $g(X)\in\mathbb{Q}[X]$. By Gauss's Lemma, $g(X)\in\mathbb{Z}[X]$. Therefore the constant term $\operatorname{Norm}_\mathbb{Q}^{\mathbb{Q}(j)} j$ of $f(X)$ divides $a_{0,0}$.
\end{proof}

For the norms of some CM $j$-invariants, see \cite{Gross-Zagier} Table $1$, and \cite{Berwick_CM}.

\begin{proposition}\label{prop.a_0_0_2}
	Let $\ell>5$ be a prime and let $a_{0,0}$ be the coefficient of $\Phi_\ell(X,Y)$. Then the following is true.
	\begin{enumerate}
		\item $a_{0,0}\equiv 0 \operatorname{mod} 2^{15}3^{3}5^{3}$.
		\item If Conjecture \ref{conj.div} is true for all $a_{m,n}$ with $(m,n)\ne (0,0)$ such that $\ell+1>m+n$, then Conjecture \ref{conj.div} is true for $a_{0,0}$.
	\end{enumerate}
\end{proposition}

\begin{proof}
	
	By Proposition \ref{prop.verif}, we can assume that $\ell>\bd$. Fix a prime $\ell>\bd$.
	
	{\it The $2$-divisibilities.}
	The prime $\ell$ satisfies exactly one of the conditions (a),(b),(c) below. (a) $\ell\equiv 3 \;(\text{resp. } 7)\operatorname{mod} 8$. Let $x:=4\; (\text{resp. } 2)$. (b) $\ell\equiv 5 \;(\text{resp. } 13, 21, 29)\operatorname{mod} 32$. Let $x:=5\; (\text{resp. } 1,3,7)$. (c) 
	$$\ell\equiv 1 \;(\text{resp. } 9,17,25,33,41,49,57,65,73,81,89,97,105,113,121)\operatorname{mod} 128.$$ Let 
	$$x:=9\; (\text{resp. } 27,15,19,25,11,1,3,23,5,17,13,7,21,31,29).$$ 
	Assume that $\ell-x^2>0$. Write $\ell-x^2$ in the form $ny^2$ where $n,y\in\mathbb{N}$ and $n$ is squarefree. Our choice of $x$ makes $n\equiv 3\operatorname{mod} 8$ in all cases. Indeed, in case (a), this is clear. In case (b), we have that $ny^2\equiv 12\operatorname{mod} 32$. The possible residues of $y^2$ modulo $32$ are the following: $0, 1, 4, 9, 16, 17, 25$. If $y^2\equiv0\operatorname{mod} 32$, then $0\equiv12\operatorname{mod} 32$, a contradiction. If $y^2\equiv16\operatorname{mod} 32$, then $4n\equiv 3\operatorname{mod}8$ and hence $0\equiv 3\operatorname{mod} 4$, a contradiction. If the residue of $y^2$ modulo $32$ is one of the following: $1, 9, 17, 25$, then $4$ divides $n$. This contradicts the assumption that $n$ is squarefree. Hence $y^2\equiv 4\operatorname{mod} 32$ and $n\equiv 3\operatorname{mod} 8$. In case (c), we have that $ny^2\equiv 48\operatorname{mod} 128$. The possible residues of $y^2$ modulo $128$ are the following: $ 0, 1, 4, 9, 16, 17, 25, 33, 36, 41, 49, 57, 64, 65, 68, 73, 81, 89, 97, 100,105, 113, 121$. By an argument similar to case (b), we can derive that $y^2\equiv 16$ and thus $n\equiv 3\operatorname{mod} 8$. So $n\equiv 3\operatorname{mod} 8$ in each case.
	
	(i) Let $K:=\mathbb{Q}(\sqrt{-n})$ with ring of integers $\mathcal{O}_K$. By Theorem 9.4 in \cite{Cox}, $\ell$ splits completely in the Hilbert class field $H$ of $K$. Let $E/H$ be an elliptic curve with geometric endomorphism ring $\mathcal{O}_K$ (see \cite{SilvermanAd} p. 99 and p. 105). Let $j:=j(E)$.  Then $|\operatorname{Norm}_\mathbb{Q}^{\mathbb{Q}(j)} j|$ divides $a_{0,0}$ by Proposition \ref{prop.a_0_0}. Moreover, let $d$ be the discriminant of $K$. Then $d=-n\equiv 5\operatorname{mod} 8$. Corollary 2.5, 2) in \cite{Gross-Zagier} shows that $j=2^{15}a$ for some $a\in\mathcal{O}_{\mathbb{Q}(j)}$. Recall that $[\mathbb{Q}(j):\mathbb{Q}]=h(K)$, where $h(K)$ is the class number of $K$. Hence $2^{15h(K)}$ divides $|\operatorname{Norm}_\mathbb{Q}^{\mathbb{Q}(j)} j|$. Therefore $2^{15h(K)}$ divides $a_{0,0}$.

	(ii) Keep the notation in the proof of part (i). Let $\mathfrak{p}$ be a prime ideal in $H$ above the ideal $(2)$ in $\mathbb{Q}$ and let $v$ be the valuation on $H$ associated to $\mathfrak{p}$. Since $\ell$ splits completely in $H$, we know that $\Phi_\ell(j,j)=0$ (see the proof of Proposition \ref{prop.a_0_0}). It follows that $$v(a_{0,0})=v(-2j^{\ell+1}-\sum\limits_{\substack{0\le m,n\le \ell\\
			(m,n)\ne(0,0)}}a_{m,n}j^{m+n}).$$ Assume that Conjecture \ref{conj.div} is true for all $a_{m,n}$ with $(m,n)\ne (0,0)$ such that $\ell+1>m+n$. Because $j=2^{15}a$ with $a\in\mathcal{O}_{\mathbb{Q}(j)}$, we have that $v(a_{m,n}j^{m+n})\ge 15(\ell+1-m-n)v(2)+15(m+n)v(2)=15(\ell+1)v(2)$ for all $m,n$ such that $0\le m,n\le \ell$ and $(m,n)\ne (0,0)$. We also have that $v(-2j^{\ell+1})\ge 15(\ell+1)v(2)+v(2)$. Hence $v(a_{0,0})\ge 15(\ell+1)v(2)$ and Conjecture \ref{conj.div} is true for $a_{0,0}$.

	Assume that $\ell-x^2\le 0$ with $\ell>\bd$. Then it is easy to check that $\ell$ is one of the following: $449, 521, 761, 881$. Let $\ell=449$ (resp. $521$, $761$, $881$). Then the proof of Corollary 3.20 in \cite{Wang_thesis} shows that $\Phi_\ell(j,j)=0$ for $j=-2^{15}$ (resp. $-2^{15}$, $-2^{15} 3^3$, $-2^{15}$) since $\legendre{-11}{449}=1$ (resp. $\legendre{-11}{521}=1$, $\legendre{-19}{761}=1$, $\legendre{-11}{881}=1$). It is now clear that the $2$-divisibilities in (i) and (ii) hold for $\ell$.

	Next, we prove the divisibilities by $3$ and $5$. We keep the notation in the proof of $2$-divisibilities in the case $\ell-x^2>0$. Since the idea is similar, we only point out of the required adjustments. 
	
	{\it The $3$-divisibilities.} Note that we can assume that $\ell\equiv 2\operatorname{mod} 3$ since $a_{0,0}=0$ otherwise by Theorem \ref{thm.ito}. Let $\ell\equiv 2\operatorname{mod} 3$ and let $x:=1$. Then $n\equiv 1\operatorname{mod} 3$ and $d\equiv 2\operatorname{mod} 3$. Corollary 2.5, 2) in \cite{Gross-Zagier} shows that $j=1728+3^6 a=3^3(2^6+a)$ for some $a\in\mathcal{O}_{\mathbb{Q}(j)}$. The $3$-divisibilities in (i) and (ii) follow.
	
	{\it The $5$-divisibilities.} If $\ell\equiv 1 \;(\operatorname{resp.} 2, 3, 4)\operatorname{mod} 5$, let $x:=2\; (\operatorname{resp.} 2, 1, 1)$. Then $n\equiv 2\text{ or } 3\operatorname{mod} 5$ and $d\equiv 2\text{ or } 3\operatorname{mod} 5$. We know that $j=5^3 a$ for some $a\in\mathcal{O}_{\mathbb{Q}(j)}$ by Corollary 2.5, 2) in \cite{Gross-Zagier}. The $5$-divisibilities in (i) and (ii) follow.
\end{proof}

Our initial interest in Conjecture \ref{conj.div} was motivated by some considerations about isogenous elliptic curves (\cite{Wang_red}, \repisogeny). In particular, let us mention the following conjecture, which is implied by Conjecture \ref{conj.div}.

\begin{conjecture}\label{conj.eqvaluation}
	Let $\mathcal{O}_K$ be a discrete valuation domain with field of fractions $K$ and residue field $k$. Let $v$ be the normalized valuation of $K$. Let $E_1/K$ and $E_2/K$ be elliptic curves and assume that there exists an isogeny of degree $d$ from $E_1/K$ to $E_2/K$ where $d$ is coprime to $\operatorname{char}(k)$. Then the following is true.
	\begin{enumerate}[label=(\alph*)]
		\item Assume that $\operatorname{char}(k)=2$. If $0< v(j(E_1))< 15v(2)$, then $v(j(E_1))=v(j(E_2))$. 
		\item Assume that $\operatorname{char}(k)=3$. If $0< v(j(E_1))< 3 v(3)$, then $v(j(E_1))=v(j(E_2))$. Moreover, if $\ell\equiv 1\operatorname{mod} 3$ and $0< v(j(E_1))< \frac{9}{2} v(3)$, then $v(j(E_1))=v(j(E_2))$. 
		\item Assume that $\operatorname{char}(k)=5$. If $0< v(j(E_1))< 3v(5)$, then $v(j(E_1))=v(j(E_2))$.

	\end{enumerate}
\end{conjecture}

\begin{proposition}\label{prop.imply}
	Conjecture \ref{conj.div} (a) (resp. (b), (c)) implies Conjecture \ref{conj.eqvaluation} (a) (resp. (b), (c)). 
\end{proposition}

\begin{proof}
	We only prove that Conjecture \ref{conj.div} (a) implies Conjecture \ref{conj.eqvaluation} (a). The other statements can be proved similarly. Let $j_1:=j(E_1)$ and $j_2:=j(E_2)$. By possibly relabeling and considering the dual isogeny, assume that $v(j_1)\le v(j_2)$. By Proposition 4.12 and Corollary 4.11 in \cite{SilvermanArith}, every isogeny of degree greater than $1$ decomposes as the composition of isogenies of prime degrees. So it is enough to consider the case where $d$ is prime. Assume that $d$ is a prime. We first work with mixed characteristic $2$. Since there exists an isogeny of degree $d$ from $E_1/K$ to $E_2/K$, we know that $\Phi_d(j_1,j_2)=0$ (see Theorem 5 in Ch. 5 of \cite{Lang}).

	Assume that part (a) of Conjecture \ref{conj.div} is true so that $15v(2)(d+1-m-n)\le v(a_{m,n})$ when $d+1\ge m+n$. Assume for the sake of contradiction that $v(j_1)<v(j_2)$. Since $v(j_1)< 15v(2)$, we claim that 
	\[v(j_1^{d+1})<v(a_{m,n}j_1^mj_2^n)\] 
	for each $0\le m,n\le d+1$ such that $(m,n)\ne (d+1,0)$. Indeed, if $m+n< d+1$, then
	\begin{align*}
		(d+1)v(j_1)&=(d+1-m-n)v(j_1)+(m+n)v(j_1)\\
		&<(d+1-m-n)15v(2)+mv(j_1)+nv(j_2)\le v(a_{m,n})+mv(j_1)+nv(j_2).
	\end{align*}
	
	If $m+n>d+1$, then $$(d+1)v(j_1)<(m+n)v(j_1)\le v(a_{m,n})+mv(j_1)+nv(j_2).$$
	
	If $m+n=d+1$ and $n\ne 0$, then
	
	\[(d+1)v(j_1)=(m+n)v(j_1)<v(a_{m,n})+mv(j_1)+nv(j_2).\] 
	
	Hence the claim is true. It follows that $v(j_1^{d+1})=v(\sum\limits_{0\le m,n\le d+1}a_{m,n}j_1^mj_2^n)= v(\Phi_d(j_1,j_2))$ where $\Phi_d(j_1,j_2)=0$. However $v(j_1)<15v(2)$ by the assumption in Conjecture \ref{conj.eqvaluation} (a). This is a contradiction. So $v(j_1)=v(j_2)$.

	Now we consider the equicharacteristic $2$ case. Let $\widetilde{\Phi}_d(X,Y)\in\mathbb{F}_2[X,Y]$ denote the $d$-th classical modular polynomial $\Phi_d(X,Y)\in\mathbb{Z}[X,Y]$ modulo $2$. Since there exists an isogeny of degree $d$ from $E_1/K$ to $E_2/K$, we know that $\widetilde{\Phi}_d(j_1,j_2)=0$ (see \cite{Moreno} p. 202). By the same arguments as in the mixed characteristic $2$ case, the conclusions follow.  
\end{proof}

\begin{proposition}
	Conjecture \ref{conj.eqvaluation} is true for $d\le 400$.
\end{proposition}

\begin{proof}
	This follows from Propositions \ref{prop.verif} and \ref{prop.imply}.
\end{proof}

\bibliographystyle{plain}
\nocite{*}
\bibliography{ref_mod.bib}

\end{document}